\documentclass[final,3p,times,12pt]{elsarticle}

\usepackage{amsmath,amsthm,amssymb,amsfonts,ascmac,bm,empheq,mathtools}
\newtheorem{thm}{Theorem}[section]
\newtheorem{lem}[thm]{Lemma}

\newtheorem{conj}[thm]{Conjecture}
\newtheorem{cor}[thm]{Corollary}
\newdefinition{rmk}{Remark}[section]
\numberwithin{equation}{section}
\numberwithin{thm}{section}

\begin{document}

\begin{frontmatter}



\title{Zeros of exceptional orthogonal polynomials and the maximum of the modulus of an energy function}


\author{Yu Luo\corref{cor1}}
\ead{luo.yu.68e@st.kyoto-u.ac.jp}
\cortext[cor1]{Corresponding author}

\address{Department of Applied Mathematics and Physics, Graduate School of Informatics, Kyoto University, Yoshida-Honmachi, Sakyo-ku, Kyoto 606-8501, Japan}

\begin{abstract}
We propose a new property of the zeros of exceptional orthogonal polynomials. 
It has been known that exceptional orthogonal polynomials (XOP) have both real and complex zeros. By fixing $m$ variables at the imaginary parts of the complex zeros of XOP, we find that in some cases the modulus of the energy function of a many-particle system attains its maximum at the zeros of XOP. We give a sufficient condition for this result with respect to the denominators of the weight function of XOP.
\end{abstract}


\end{frontmatter}


\section{Introduction}

\par
Exceptional orthogonal polynomial systems (XOPS) have been extensively studied these years. 
They differ from the classical ones (Hermite, Laguerre and Jacobi) in that there are a finite 
number of degrees do not exist in their degree sequence. The number of the missing degrees 
is called the codimension of the corresponding XOPS. In spite of the absence of degrees XOPS 
form the basis of a weighted Hilbert space, and they are also eigenfunctions of a second-order 
differential operator which has rational instead of polynomial coefficients. The most typical examples 
of XOPS can be found in \cite{key4,key5} (exceptional Hermite polynomials), \cite{key6,key11} 
(exceptional Laguerre polynomials) and \cite{key7,key24} (exceptional Jacobi polynomials).

\par
It has recently been shown that every XOPS can be obtained by applying a finite sequence of 
Darboux transformations to a classical orthogonal polynomial system (COPS) \cite[Theorem 1.2.]{key3}. 
This places on safe ground the constructive approach to a full classification of XOPS. In this 
classification all the XOPS fall into Hermite, Laguerre and Jacobi type, respectively, corresponding to 
the support $I$ and weight function $\hat{\omega}(x)$ shown in table 1,
\begin{table}[hbtp]
\begin{center}
\begin{tabular}{ccc} \hline
XOPS & $\hat{\omega}(x)$ & $I$ \\ \hline
Hermite type  & $e^{-x^2}/\eta_{H}(x)^2$ & $(-\infty,\infty)$ \\ 
Laguerre type & $e^{-x}x^{\alpha}/\eta_{L}(x)^2$ & $(0,\infty)$ \\ 
Jacobi type  & $(1-x)^{\alpha}(1+x)^{\beta}/\eta_{J}(x)^2$ & $(-1,1)$ \\ \hline
\end{tabular}
\caption{Classifiction of XOPS}
\end{center}
\end{table}
where $\eta_{H}$, $\eta_{L}$, $\eta_{J}$ are real-valued polynomials which are non-vanishing on $I$. 
It immediately follows that the XOPS return to the classical one as long as the denominator of $\hat{\omega}(x)$ 
is a constant. 

\par
The zeros of XOPS are divided into two groups: regular zeros which lie in the domain of ort-hogonality, and 
exceptional zeros (usually complex) which lie in the exterior of the domain. A conjecture considering the location 
of zeros of exceptional orthogonal polynomials was drafted as follow:
\begin{conj}[A. B. Kuijlaars and R. Milson, \cite{key8}] \textit{The regular zeros of exceptional orthogonal polynomials have the same asymptotic behavior as the zeros of their classical counterpart. The exceptional zeros converge to the zeros of the denominator polynomial $\eta(x)$}.
\end{conj}

\par
Moreover, properties like the location and asymptotic behavior of zeros of exceptional Hermite polynomials are described by A. B. Kuijlaars and R. Milson \cite{key8}, of exceptional Laguerre and Jacobi polynomials by C. L. Ho, R. Sasaki \cite{key9} and D. G\'omez-Ullate, M. Garc\'ia-Ferrero, R. Milson \cite{key10}. It concludes that the zeros of exceptional orthogonal polynomials usually share similar properties as their classical counterparts, especially for the regular zeros.

\par
Below we revisit an energy problem by making use of properties of exceptional orthogonal polynomials. Considering the maximum of the following energy function
\begin{equation}
T_{\omega}(x_1,\cdots,x_n)=\prod^{n}_{j=1}\omega(x_j)\prod_{1\leq i<j\leq n}|x_i-x_j|^2,
\end{equation}
where the $n$ points $x_1$, $\cdots$, $x_n$ lie on a compact set $E$. 
In the case of $\omega(x)=1$, I. Schur showed that the maximum of $T_{\omega}$ is obtained at the zeros of certain orthogonal polynomials \cite{key22}. If $\omega(x)$ takes a classical weight, namely with Hermite weight $\omega(x)=e^{-x^2}$, with Laguerre weights $\omega_{\alpha}(x)=x^{\alpha}e^{-x}$, with Jacobi weights $\omega_{\alpha,\beta}(x)=(1-x)^{\alpha}(1+x)^{\beta}$, then the maximum of $T_{\omega}$ is attained at the zeros of orthogonal polynomials corresponding to $\omega$, $\omega_{\alpha-1}$, $\omega_{\alpha-1,\beta-1}$, respectively \cite{key17}. Results for the zeros of general orthogonal polynomials can be found in \cite{key23}. In addition, \'A. P. Horv$\acute{a}$th proved that the set of regular zeros of exceptional Hermite polynomials is the solution of the maximum problem with respect to the weight $\hat{\omega}(x)P^{2}_{m}(x)$, where $\hat{\omega}(x)$ is the weight of exceptional Hermite polynomials, $P_{m}(x)$ is a polynomial whose zeros are the exceptional zeros of an exceptional Hermite polynomial of codimension $m$ \cite{key19}. Similar results have also been reported in the cases of the so-called $X_m$-Laguerre polynomials and $X_m$-Jacobi polynomials \cite{key18}.

\begin{rmk}
As is pointed out in \cite{key23}, $T_{\omega}$ is called an energy function in light of its potential theoretic background. In fact, taking the logarithm in $(1.1)$, the maximization problem of $(1.1)$ is equivalent to the minimization problem of the following function
$$
-\text{log}\big(T_{\omega}\big)=\sum^{n}_{j=1}\text{log}\frac{1}{\omega(x_i)}+\sum_{1\leq i<j\leq n}\text{log}\frac{1}{|x_i-x_j|^2}.
$$
The second summation in the right hand side can be interpreted as the energy of a system of $n$ like-charged particles located at the points $\{x_i\}^{n}_{i=1}$, where the repelling force between two particles is proportional to the reciprocal of the square of the distance between them. The first summation refers to the total external potential of this system. Thus, $
-\text{log}\big(T_{\omega}\big)$ is the total energy of this n-particle system.
\end{rmk}
\par
In this paper we investigate the maximum of the energy function $(1.1)$ with respect to $\omega(x)=\hat{\omega}(x)p(x)$, where $p(x)$ is the coefficient of the following second-order differential equation satisfied by the exceptional orthogonal polynomials with respect to $\hat{\omega}(x)$
\begin{equation}
p(x)y''(x)+q(x)y'(x)+r(x)y(x)=\lambda y(x),
\end{equation}
the prime denotes derivative with respect to $x$, $y'(x)=dy(x)/dx$. Note that $p(x)$, $q(x)$, $r(x)$ are rational functions satisfying the following conditions:
\begin{equation}
\text{deg}(p(x))\leq 2,\hspace{2mm} \text{deg}(q(x))\leq 1,\hspace{2mm} \text{deg}(r(x))\leq 0,
\end{equation}
only in the case when the solutions of $(1.2)$ are classical orthogonal polynomials, $p(x)$, $q(x)$, $r(x)$ return to polynomials. Here the degree of a rational function $f(x)=P(x)/Q(x)$, where $P(x)$ and $Q(x)$ are both polynomials, is given by
$$
\text{deg}(f(x))=\text{deg}(P(x))-\text{deg}(Q(x)).
$$

\par
In the next section, we give a brief introduction on the definition and properties of XOPS, and the so-called Stieltjes-Calogero type relations of the zeros of polynomial solutions of any linear second order differential equations. In particular, we derive the Stieltjes-Calogero type relations of the zeros of XOPS which can be used to prove our main results.
In section 3, some examples on the observations of the relationship between the zeros of certain XOPS and the energy function of an electrostatic model is provided. In section 4, we prove our main results, which provide as a sufficient condition for the modulus of an energy function to attain its maximum at the zeros of XOPS. Finally, section 5 concludes this paper.

\section{Preliminaries}
\par
Our main results in section 4 will be proved using the properties such as second-order differential equations of exceptional orthogonal polynomials and the Stieltjes-Calogero type relations. To this end we introduce some basic knowledge about these properties.

\subsection{Exceptional Orthogonal Polynomials}

\par
A sequence of polynomials $\{P_{n}(x)\}_{n\in\mathbb{N}}$ is called an orthogonal polynomial sequence if it satisfies
$$
\int_{I}P_{m}(x)P_{n}(x)\omega(x)dx=h_{n}\delta_{mn},\hspace{2mm} \text{deg}(P_{n}(x))=n,
$$
where $\omega(x)$ is the weight function, $I$ is called the interval of orthogonality and $\delta_{mn}$ is Kronecker's delta. 
If $h_n=1$, then the polynomials $\{P_{n}(x)\}_{n\in\mathbb{N}}$ are orthonormal. Note that the weight function $\omega(x)$ should be continuous and positive on $I$ such that the moments exist.
$$\mu_n:=\int_{I}x^{n}\omega(x)dx, \hspace{2mm} n=0,1,2,\cdots$$

\par
The most extensively studied families of orthogonal polynomials are the classical orthogonal polynomials (named by Hermite, Laguerre and Jacobi), which are the eigenfunctions of certain second-order linear differential operators. The eigenvalue equation, which takes the form of $(1.2)$, can be rewritten as the well known Sturm-Liouville type equation
$$
(P(x)y'(x))'+R(x)y(x)=\lambda\omega(x)y(x),
$$
where $P(x)=\omega(x)p(x)$, $R(x)=\omega(x)r(x)$. The weight function $\omega(x)$ satisfies the Pearson equation
\begin{equation}
\big(p(x)\omega(x)\big)'=q(x)\omega(x)
\end{equation}
and the conditions
\begin{equation}
p(x)\omega(x)x^k=0,\hspace{2mm} k\in\mathbb{N}
\end{equation}
on the boundary of the interval $I$.

\par
The exceptional orthogonal polynomial system generalizes classical orthogonal polynomial system in that it allows gaps in the polynomial sequence while preserving as eigenfunctions of a Sturm-Liouville problem \cite{key3}. Which means the exceptional weights $\hat{\omega}(x)$'s still satisfy the Pearson equation $(2.1)$ and the boundary conditions $(2.2)$. As a result of the missing degrees, coefficients of the second-order differential equation $(1.2)$ appear to be rational functions, which implies the existence of poles. Specifically, the second-order differential equations whose solutions are the three types exceptional orthogonal polynomials describe as \cite{key5,key6,key7}
\begin{equation}
H''_{n}(x)-2\bigg(x+\frac{\eta'_{H}(x)}{\eta_{H}(x)}\bigg)H'_{n}(x)+\bigg(\frac{\eta''_{H}(x)}{\eta_{H}(x)}+2x\frac{\eta'_{H}(x)}{\eta_{H}(x)}+2n-2k-2u^{H}_{\mathcal{F}}\bigg)H_{n}(x)=0,
\end{equation}

\begin{equation}
\begin{split}
& xL''_{n}(x)+\bigg(\alpha+k'+1-x-2x\frac{\eta'_{L}(x)}{\eta_{L}(x)}\bigg)L'_{n}(x) \\
& +\bigg(x\frac{\eta''_{L}(x)}{\eta_{L}(x)}+(x-\alpha-k')\frac{\eta'_{L}(x)}{\eta_{L}(x)}+n-k_{1}-u^{L}_{\mathcal{F}}\bigg)L_{n}(x)=0, \quad \alpha>-1, k'>0,
\end{split}
\end{equation}

\begin{equation}
\begin{split}
& (1-x^2)P''_{n}(x)+\bigg(\beta-\alpha-2k'_{2}-(\alpha+\beta+2k'_{1}+2)x-2(1-x^2)\frac{\eta'_{J}(x)}{\eta_{J}(x)}\bigg)P'_{n}(x) \\
& +\bigg((1-x^2)\frac{\eta''_{L}(x)}{\eta_{L}(x)}+\big[\alpha-\beta+2k'_{2}+(2k'_{1}+\alpha+\beta)x\big]\frac{\eta'_{J}(x)}{\eta_{J}(x)}+\lambda(n-u^{J}_{\mathcal{F}})-\lambda(k'_1)\bigg)P_{n}(x)=0, \\
& \alpha,\beta>-1, k'_1+k'_2>0,
\end{split} 
\end{equation}
where $H_{n}(x)$, $L_{n}(x)$, $P_{n}(x)$ denote exceptional Hermite, Laguerre, Jacobi polynomials of degree $n$, respectively. $\eta_{H}(x)$, $\eta_{L}(x)$, $\eta_{J}(x)$ are polynomials whose degrees coincide with the codimension of the related exceptional orthogonal polynomial systems, $\alpha$, $\beta$, $k$, $k'$, $k_1$, $k'_1$, $k'_2$, $u^{H}_{\mathcal{F}}$, $u^{L}_{\mathcal{F}}$, $u^{J}_{\mathcal{F}}$ are certain constants and $\lambda(x)$ is a real-valued function, we shall omit the details about these functions and constants in this paper for the convenience of discussion.

\subsection{Stieltjes-Calogero type relations}

\par
There are many literatures considering the Stieltjes-Calogero type relations for zeros of orthogonal polynomials, the most famous result among which was obtained by T. J. Stieltjes\cite{key1} as follow
$$
\sum^{n}_{k=1,k\neq j}\frac{1}{x_j-x_k}=x_j,
$$
where $x_1$, $x_2$, $\cdots$, $x_n$ are zeros of Hermite polynomial of degree $n$. Stieltjes noted that this result implies an appealing interpretation of the location of zeros of Hermite polynomials as equilibrium positions of a simple one-dimensional n-particle problem. Moreover, He obtained similar relations for zeros of Laguerre and Jacobi polynomials thereafter. Interest in such kind of relations was revived by the work of Calogero and co-workers on integrable many-body systems\cite{key12,key14,key15}. Since then substantial efforts have been made on finding the Stieltjes-Calogero type relations for the purpose of revealing the relationship between zeros of polynomial systems and certain many-body systems.

\par
To the best of the author's knowledge, the existing most generic method of obtaining this kind of relations was described in \cite{key20}. We apply this method to give some nontrivial results in the proceeding part.
Let
$$S_{m,j}:=\sum^{n}_{k=1,k\neq j}\frac{1}{(x_j-x_k)^m},$$
if it satisfies that
$$\sum^{n}_{k=1,k\neq j}\frac{1}{(x_j-x_k)^m}=f(x_j),$$
where $f(x_j)$ is a rational function about $x_j$, then the above formula is called a Stieltjes-Calogero type relation.

\par
Consider an $n$-th order differential equation，
\begin{equation}
\sum^{n}_{i=0}A_{i}(x)y^{(n-i)}(x)=f(x),
\end{equation}
where $A_i(x)$ and $f(x)$ belong to $C^{\infty}(-\infty,\infty)$．
Suppose that $(2.6)$ has a monic polynomial solution $y(x)$ with simple roots:
$$y(x)=\prod^{n}_{i=1}(x-x_i),$$
then let $y_j(x)$ be defined as $y(x)=(x-x_j)y_j(x)$, i.e.
$$y_j(x)=\prod^{n}_{i=1,i\neq j}(x-x_i).$$
It follows that
$$y^{(r)}(x_j)=ry^{(r-1)}_{j}(x_j), \hspace{2mm} r\geq 1,$$
so that $(2.6)$ becomes, after division by $y'(x)$ and evaluation at $x=x_j$，
\begin{equation}
\sum^{n-1}_{i=0}(n-i)A_{i}(x_j)\frac{y^{(n-i-1)}_{j}(x_j)}{y_j(x_j)}=\frac{f(x_j)}{y'(x_j)}.
\end{equation}
$S_{1,j}$ can easily be obtained by observing the right hand side of the following formula 
$$S_{1,j}=\frac{y'_j(x)}{y_{j}(x)}\big|_{x=x_j}$$
thus the other terms immediately follow by differentiating at $x=x_j$
$$\bigg(\frac{y'_j(x)}{y_{j}(x)}\bigg)^{(s)}\big|_{x=x_j}=(-1)^{s}s!S_{s+1,j}, \hspace{2mm} s=0,1,2,\cdots.$$
In light of the above formula $S_{r,j}$($r=2,3,\cdots$) can be derived by analyzing a new function $Z_r(x)$
$$Z_r(x):=\frac{y^{(r)}_{j}(x)}{y_{j}(x)}$$
where $Z_r(x)$ satisfies a recurrence relation
$$Z_{r+1}(x)=Z'_{r}(x)+Z_{1}(x)Z_{r}(x),$$
and the initial condition
$$Z_{1}(x_j)=S_{1,j}.$$

Immediately we can rewrite $(2.7)$ as
\begin{equation}
\sum^{n-1}_{i=0}(n-i)A_{i}(x_j)Z_{n-i-1}(x_j)=\frac{f(x_j)}{y'(x_j)}.
\end{equation}

In the case of exceptional orthogonal polynomials, a second-order differential equation with rational coefficients in the shape of $(1.2)$ was satisfied, one can easily obtain
\begin{eqnarray}
S_{1,j} &=& -\frac{q(x_j)}{2p(x_j)}, \\
S_{2,j} &=& \frac{2[p'(x_j)+q(x_j)]S_{1,j}+[q'(x_j)+r(x_j)]}{3p(x_j)}+S^2_{1,j}, \\
S_{3,j} &=& -\frac{1}{8p(x_j)}\bigg\{3[2p'(x_j)+q(x_j)][S^{2}_{1,j}-S_{2,j}]+\nonumber  \\
& & 2[p''(x_j)+2q'(x_j)+r(x_j)]S_{1,j}\bigg\}+\frac{3}{2}S_{1,j}S_{2,j}-\frac{1}{2}S^{3}_{1,j},
\end{eqnarray}
and $S_{4,j}$, $S_{5,j}$, $\cdots$, by inductively computing $Z_{r}(x)$, $r=2,3,\cdots$, and differentiating on $(1.2)$.

\par
Making use of the above method, we obtain the following properties on the zeros of exceptional orthogonal polynomials according to the second-order differential equations $(2.3)$, $(2.4)$ and $(2.5)$.

Let $x_1$, $\cdots$, $x_n$ denote the $n$ zeros of the exceptional Hermite polynomial of degree $n$, then the Stieltjes-Calogero type relations of $x_1$, $\cdots$, $x_n$ follow
\begin{eqnarray*}
S_{1,j} &=& x_j+\frac{\eta'_{H}(x_j)}{\eta_{H}(x_j)}, \\
S_{2,j} &=& \frac{2}{3}(n-1-k-u^{H}_{\mathcal{F}})-\frac{1}{3}\bigg[x^2_j+\frac{\eta''_{H}(x_j)}{\eta_{H}(x_j)}-\bigg(\frac{\eta'_{H}(x_j)}{\eta_{H}(x_j)}\bigg)^2\bigg], \\
S_{3,j} &=& \frac{1}{2}x_j.
\end{eqnarray*}

Let $x_1$, $\cdots$, $x_n$ denote the $n$ zeros of the exceptional Laguerre polynomial of degree $n$, then the Stieltjes-Calogero type relations of $x_1$, $\cdots$, $x_n$ follow
\begin{eqnarray*}
S_{1,j} &=& -\frac{\alpha+1+k'-x_j}{2x_j}+\frac{\eta'_{L}(x_j)}{\eta_{L}(x_j)}, \\
S_{2,j} &=& -\frac{1}{12}\bigg\{\frac{(\alpha+1+k')(\alpha+5+k')}{x^2_j}-\frac{2(2n+\alpha+1+k'-2k'_1-2u^{L}_{\mathcal{F}}+2\frac{\eta'_{L}(x_j)}{\eta_{L}(x_j)})}{x_j} \\
& & \quad\quad\quad +1+4\frac{\eta''_{L}(x_j)}{\eta_{L}(x_j)}-4\bigg(\frac{\eta'_{L}(x_j)}{\eta_{L}(x_j)}\bigg)^2\bigg\}.
\end{eqnarray*}

Let $x_1$, $\cdots$, $x_n$ denote the $n$ zeros of the exceptional Jacobi polynomial of degree $n$, then the Stieltjes-Calogero type relations of $x_1$, $\cdots$, $x_n$ follow
\begin{eqnarray*}
S_{1,j} &=& -\frac{\alpha-\beta+2k'_2+(\alpha+\beta+2+2k'_1)x_j}{2(1-x^2_j)}+\frac{\eta'_{J}(x_j)}{\eta_{J}(x_j)}. 
\end{eqnarray*}

\begin{rmk}
Only the first several terms of these relations are listed here, the other terms, which tend to be more complicated (although some special terms may have elegant forms like $S_{3,j}$ for the zeros of exceptional Hermite polynomials), can be easily computed using this method. Notice that in the case of classical orthogonal polynomials all the terms containing $k$, $k'$, $k_1$, $k'_1$, $k'_2$, $u^{a}_{\mathcal{F}}$, $\eta'_{a}(x_j)/\eta_{a}(x_j)$ and $\eta''_{a}(x_j)/\eta_{a}(x_j)$ $(a=H,L,J)$ disappear automatically. 
\end{rmk}

\section{Examples}

\par
In this section we provide some examples which give evidence for our main result considering the case of exceptional Hermite polynomials. The exceptional Hermite polynomials are defined upon Wronskian determinants whose entries are Hermite polynomials according to a double partition \cite{key4}. Let $\lambda=(\lambda_1,\cdots,\lambda_r)$ be a non-decreasing sequence of non-negative integers
$$
0\leq \lambda_1\leq\lambda_2\leq\cdots\leq\lambda_r,	
$$
we call $\lambda$ a double (or even) partition if $r$ is even and $\lambda_{2i-1}=\lambda_{2i}$, $i=1,\cdots,r/2$. 
The exceptional Hermite polynomials with respect to $\lambda$ are defined as
$$
H^{(\lambda)}_{n}=\text{Wr}[H_{\lambda_1},H_{\lambda_2+1},\cdots,H_{\lambda_{r}+r-1},H_{n-|\lambda|+r}], \quad n-|\lambda|+r\in\mathbb{N}\backslash \{\lambda_1,\lambda_2+1,\cdots,\lambda_r+r-1\},
$$
where \text{Wr} denotes the Wronskian determinant, $H_{j}$ is the $j$th Hermite polynomial and $|\lambda|=\sum^{r}_{i}\lambda_{i}$. 
From this definition it is clear that deg$H^{(\lambda)}_{n}(z)=n$.
Recall from table 1 the weight function of exceptional Hermite polynomials is $\hat{\omega}_{H}(z)=e^{-z^2}/\eta^{2}_{H}(z)$,
where we can now give $\eta_{H}$ as
$$
\eta_{H}:=\eta^{(\lambda)}_{H}=\text{Wr}[H_{\lambda_1},H_{\lambda_2+1},\cdots,H_{\lambda_{r}+r-1}].
$$
It is known that $\eta_{H}$ has no zeros on the real line when $\lambda$ is a double partition \cite{key4}, hence $\hat{\omega}_{H}$ is a well-defined weight function on the real line. Since deg$\eta_{H}=|\lambda|$, $\eta_{H}$ has $|\lambda|$ complex zeros. According to theorem 2.3 of \cite{key8}, if all the zeors of $H^{(\lambda)}_{n}$ are simple then the exceptional (complex) zeros converge to the zeros of $\eta_{H}$.
%which implies that $H^{(\lambda)}_{n}$ has exactly $|\lambda|$ exceptional (complex) zeros, thus $n-|\lambda|$ regular (real) zeros.

\par
The problem described in the introduction is to find the maximum value of
$$
T_{\omega}(x_1,\cdots,x_n)=\prod^{n}_{j=1}\omega(x_j)\prod_{1\leq i<j\leq n}|x_i-x_j|^2,
$$
where $\omega=\hat{\omega}p$, specifically $\omega(x)=e^{-x^2}/\eta^2_H(x)$ in the current case. Let $Z=\{z_1,\cdots,z_n\}$ be the set of zeros of $H^{(\lambda)}_{n}(z)$.
In order to check whether $|T_{\omega}|$ has a maximum value at $Z$ or not, define
$$f(z)=\bigg|\frac{T_{\omega}(z_1+z,\cdots,z_n+z)}{T_{\omega}(z_1,\cdots,z_n)}\bigg|,$$
for different partition $\lambda$ we observe the value of $f(z)$ around $z=0$.

\paragraph{Example 1}
When $\lambda=(1,1,1,1)$, $\eta_{H}=\text{Wr}[H_1,H_2,H_3,H_4]$, the associated exceptional Hermite polynomials are
$$
H^{(\lambda)}_{n}=\text{Wr}[H_1,H_2,H_3,H_4,H_n], \quad n\notin\{1,2,3,4\}.
$$ 
Let $n=8$, then $H^{(\lambda)}_{n}(z)$ has 4 complex zeros and 4 real zeros, $Z=\{z_1,\cdots,z_8\}$. Numerical results show that $z=0$ is a saddle point of $f(z)$ when $z\in\mathbb{C}$ (since $T_{\omega}(z_1+z,\cdots,z_n+z)$ is a holomorphic function, according to the maximum modulus principle the modulus $|T_{\omega}(z_1+z,\cdots,z_n+z)|$ cannot exhibit a true local maximum within the domain). Nevertheless, if $z\in\mathbb{R}$, $f(z)$ attains its maximum at $z=0$.

\paragraph{Example 2}
For $\lambda=(1,1,3,3)$, $\eta_{H}=\text{Wr}[H_1,H_2,H_5,H_6]$, the associated exceptional Hermite polynomials are
$$
H^{(\lambda)}_{n}=\text{Wr}[H_1,H_2,H_5,H_6,H_{n-4}], \quad n-4\notin\{1,2,5,6\}.
$$ 
Let $n=8$, then $H^{(\lambda)}_{n}(z)$ has 6 complex zeros and 2 real zeros, $Z=\{z_1,\cdots,z_{8}\}$. Again, it follows numerically that $z=0$ is a saddle point of $f(z)$ when $z\in\mathbb{C}$ and a maximum point of $f(z)$ if $z\in\mathbb{R}$.

\begin{rmk}
Example 1 and example 2 show that in some cases $Z$ is a saddle point of $|T_{\omega}|$ while at the same time a maximum point of $|T_{\omega}|$ if all the imaginary parts of $z_i$'s are fixed. However, this phenomenon does not arise for all cases. 
\end{rmk}

\paragraph{Example 3}
For $\lambda=(2,2,3,3)$, $\eta_{H}=\text{Wr}[H_2,H_3,H_5,H_6]$, the associated exceptional Hermite polynomials are
$$
H^{(\lambda)}_{n}=\text{Wr}[H_2,H_3,H_5,H_6,H_{n-6}], \quad n\notin\{2,3,5,6\}.
$$ 
Let $n=10$, then $H^{(\lambda)}_{n}(z)$ has 8 complex zeros and 2 real zeros, $Z=\{z_1,\cdots,z_{10}\}$. In this case one can observe from the numerical simulation of $f(z)$ that $z=0$ is neither a maximum point nor a saddle point of $f(z)$, hence $|T_{\omega}|$ has no maximum at $Z$.

\section{Main results}
\par
The examples in section 3 indicate that in some special cases the modulus of the energy function $T_{\omega}$ attains its maximum at the zeros of a XOP.
In this section we investigate under what kind of conditions it leads to these special cases, 
the main results are concluded as Theorem 4.4. Before proving the main theorems we will give several lemmas which consider the positive definiteness of a matrix and the uniqueness of the maximum point of $T_{\omega}$.

\begin{lem}
An Hermitian strictly diagonally dominant matrix with real positive diagonal entries is positive definite.
\end{lem}

\begin{proof}
Let $A$ denote an Hermitian strictly diagonally dominant matrix with real positive diagonal entries, then it follows from the Gershgorin circle theorem that all the eigenvalues of $A$ are positive, which implies that $A$ is positive definite.
\end{proof}

\begin{lem}[Uniqueness of the maximum point of $T_{\omega}$] Let $\omega$ be a non-nagative, continuous weight on $I\subset\mathbb{R}$ such that $\text{log}\omega$ is concave, i.e., $(\text{log}\omega(x))''\leq 0$, $\forall x\in I$, then the maximum point of $T_{\omega}$ is unique.

\end{lem}

\begin{proof}
Assume that $\{a_i\}^{n}_{i=1}$ and $\{b_i\}^{n}_{i=1}$ are maximum points of $T_{\omega}$ enumerated in increasing order, let $c_i=(a_i+b_i)/2$. We consider the value of $T_{\omega}$ at the point $\{c_i\}^{n}_{i=1}$. Rewrite $T_{\omega}$ as
$$
T_{\omega}(x_1,\cdots,x_n)=\prod^{n}_{j=1}\omega(x_j)\prod_{1\leq i<j\leq n}|x_i-x_j|^2=\prod_{1\leq i<j\leq n}|x_i-x_j|^2[\omega(x_i)\omega(x_j)]^{\frac{4}{n(n-1)}},
$$
then because of the ordering of the points and the log-concavity of $\omega$, using the arithmetic-geometric mean inequality,
$$
\text{log}\big\{|c_i-c_j|^2[\omega(c_i)\omega(c_j)]^{\frac{4}{n(n-1)}}\big\} = \text{log}|c_i-c_j|^2+\frac{4}{n(n-1)}\big[\text{log}\omega(c_i)+\text{log}\omega(c_j)\big]
$$
$$
\hspace{-8.5mm}=\text{log}\big(|\frac{a_i-a_j}{2}|+|\frac{b_i-b_j}{2}|\big)^2+\frac{4}{n(n-1)}\big[\text{log}\omega(\frac{a_i+b_i}{2})+\text{log}\omega(\frac{a_j+b_j}{2})\big]
$$
$$
\geq \text{log}|a_i-a_j||b_i-b_j|+\frac{2}{n(n-1)}\big[\text{log}\omega(a_i)+\text{log}\omega(b_i)+\text{log}\omega(a_j)+\text{log}\omega(b_j)\big]
$$
$$
\hspace{-1cm} =\frac{1}{2}\text{log}\big\{|a_i-a_j|^2[\omega(a_i)\omega(a_j)]^{\frac{4}{n(n-1)}}\big\}+\frac{1}{2}\text{log}\big\{|b_i-b_j|^2[\omega(b_i)\omega(b_j)]^{\frac{4}{n(n-1)}}\big\}
$$
where the equality holds if and only if $a_i=b_i$, $i=1,\cdots,n$, which establishes the uniqueness.
\end{proof}

\par
As it is pointed out in \cite{key21} that Stieltjes shows (when $\omega$ is a classical weight) $-\text{log}T_{\omega}(x_1,\cdots,x_n)$ attains a minimum when $x_1$, $\cdots$, $x_n$ are the zeros of the corresponding classical orthogonal polynomial. However, according to the observation of an author of \cite{key21} Stieljes does not explicitly show that this position is a minimum (even though he explicitly mentions that it is a minimum). Here we reformulate these results as the following theorem and give an explicit proof.

\begin{thm}
Let $\omega(x)=\hat{\omega}(x)p(x)$, where $\hat{\omega}(x)$ takes a classical weight, namely $\omega(x)=e^{-x^2}$ for Hermite polynomials, or $\omega(x)=x\cdot x^{\alpha}e^{-x}$ for Laguerre polynomials, or $\omega(x)=(1-x^2)\cdot (1-x)^{\alpha}(1+x)^{\beta}$ for Jacobi polynomials. Then in the domain $I$ with respect to $\hat{\omega}(x)$, the energy function $T_{\omega}$ attains its maximum at the set of zeros of the corresponding orthogonal polynomials. 
\end{thm}

\begin{proof}
Let $x_1$, $\cdots$, $x_n$ denote the zeros of 
classical orthogonal polynomial of degree $n$ with respect to $\hat{\omega}(x)$ (with $\eta(x)=1$),
\begin{eqnarray*}
\frac{\partial\text{log}T_{\omega}(y_1,\cdots,y_n)}{\partial y_i} &=& \big(\frac{\omega'}{\omega}\big)(y_i)+\sum^{n}_{k=1,k\neq i}\frac{2}{y_i-y_k}, \\
&=& \frac{\hat{\omega}'(y_i)}{\hat{\omega}(y_i)}+\frac{p'(y_i)}{p(y_i)}+\sum^{n}_{k=1,k\neq i}\frac{2}{y_i-y_k}, 
\end{eqnarray*}
it follows from Pearson equation $(2.1)$ and the Stieltjes-Calogero type relation $(2.9)$
$$
\frac{\hat{\omega}'(x)}{\hat{\omega}(x)}+\frac{p'(x)}{p(x)}=\frac{q(x)}{p(x)}, \hspace{4mm} S_{1,i}=\sum^{n}_{k=1,k\neq i}\frac{1}{x_i-x_k}=-\frac{q(x_i)}{2p(x_i)}.
$$
that 
$$
\frac{\partial\text{log}T_{\omega}(x_1,\cdots,x_n)}{\partial x_i}=0.
$$
Thus, $X=$\{$x_1$, $\cdots$, $x_n$\} is a critical point of the energy function $T_{\omega}$, which means $T_{\omega}$ has a local extremum at $X$. Next we consider the Hessian matrix $H$ of $(-\text{log}T_{\omega})$, 
if $H$ is positive definite at $X$, then $T_{\omega}$ has a local maximum at $X$. The off-diagonal and diagonal elements of $H$ are given by

\begin{eqnarray*}
H_{ij}=\frac{\partial^{2}(-\text{log}(T_{\omega}(y_1,\cdots,y_n)))}{\partial y_{i}\partial y_{j}} &=& -\frac{2}{(y_i-y_j)^2}, \quad i\neq j \\
H_{ii}=\frac{\partial^{2}(-\text{log}(T_{\omega}(y_1,\cdots,y_n)))}{\partial y_{i}^{2}} &=& \big(-\frac{\omega'}{\omega}\big)'(y_i)+\sum^{n}_{j=1,j\neq i}\frac{2}{(y_i-y_j)^2} \\
&=& \frac{q(y_i)p'(y_i)-p(y_i)q'(y_i)}{p(y_i)^2}+\sum^{n}_{j=1,j\neq i}\frac{2}{(y_i-y_j)^2}.
\end{eqnarray*}
Since $q(x)p'(x)-p(x)q'(x)>0$ for any $x\in\mathbb{R}$, $H$ is Hermitian, strictly diagonally dominant, and has real positive diagonal entries, thus it immediately follows that $H$ is positive definite, which means $T_{\omega}$ has a maximum value at $\{x_1, \cdots, x_n\}$. In fact, $q(x)p'(x)-p(x)q'(x)>0$ is always true in the case of classical orthogonal polynomials, denote the left hand side of the inequality as $F(x)$. For Hermite polynomials, $p(x)=1,q(x)=-2x$, $F(x)=2>0$; for Laguerre polynomials, $p(x)=x,q(x)=\alpha+1-x$, $F(x)=\alpha+1>0$ (since $\alpha>-1$); for Jacobi polynomials, $p(x)=1-x^2,q(x)=\beta-\alpha-(\alpha+\beta+2)x$, $F(x)=(\alpha+\beta+2)(1+x^2)$ $-(\beta-\alpha)2x>0$ (since $\alpha,\beta>-1$).

The uniqueness follows from the fact that each weight $\omega(x)$ is log-concave in the related domain $I$. Moreover, $T_{\omega}$ tends to zero at the boundary of the domain related to $\hat{\omega}$, thus it attains a unique maximum at $\{x_1, \cdots, x_n\}$.
\end{proof}

\par
Note that the zeros of classical orthogonal polynomials are all real, simple and distinct \cite[chapter 6.2]{key1}, which guarantees that $T_{\omega}$ dose not vanish at the set of these zeros. However, the exceptional orthogonal polynomials have complex zeros and were conjectured to have simple zeros except possibly for the zeros at $z=0$ \cite{key8}. 
Here we assume an exceptional orthogonal polynomial $P_{n+m}(z)$ of degree $n+m$ has simple zeros, 
denote the set of zeros of $P_{n+m}(z)$ as
$$
Z=\{z_1,\cdots,z_n,z_{n+1},\cdots,z_{n+m}\}
$$
where $z_1=x_1$, $\cdots$, $z_n=x_n$ are the $n$ real zeros and $z_{n+1}=x_{n+1}+i\mu_1$, $\cdots$, $z_{n+m}=x_{n+m}+i\mu_m$ are the $m$ complex zeros. Fix $\mu_1$, $\cdots$, $\mu_m$, we consider the function $T_{\omega}(Y)=T_{\omega}(y_1,\cdots,y_n,y_{n+1}+i\mu_{1},\cdots,y_{n+m}+i\mu_{m})$ with $n+m$ real variables. $T_{\omega}(Y)$ is a complex-valued function as long as $m\geq 1$, so we check the maximum value of $|T_{\omega}(Y)|^2=T_{\omega}(Y)\overline{T_{\omega}(Y)}$ instead. First, rewrite $T_{\omega}$ as
$$
T_{\omega}(Y)=\prod^{n}_{i=1}\omega(y_i)\cdot\prod^{m}_{j=1}\omega(y_{n+j}+i\mu_{j})\cdot\prod_{1\leq i<j\leq n}|y_i-y_j|^2\cdot\prod_{1\leq k<l\leq m}|y_{n+k}+i\mu_{k}-(y_{n+l}+i\mu_{l})|^2
$$
$$
\hspace{-6.8cm}\cdot\prod_{\substack{1\leq s\leq n \\ 1\leq t\leq m}}|y_{s}-(y_{n+t}+i\mu_{t})|^2,
$$
then we have
$$
|T_{\omega}(Y)|^2=\prod^{n}_{i=1}\omega^{2}(y_i)\cdot\prod^{m}_{j=1}\omega(y_{n+j}+i\mu_{j})\omega(y_{n+j}-i\mu_{j})\cdot\prod_{1\leq i<j\leq n}|y_i-y_j|^4\cdot\prod_{1\leq k<l\leq m}|(y_{n+k}-y_{n+l})^2+(\mu_{k}-\mu_{l})^2|^2
$$
$$
\hspace{-7.2cm} \cdot\prod_{\substack{1\leq s\leq n \\ 1\leq t\leq m}}|(y_{s}-y_{n+t})^2+\mu^2_{t})|^2.
$$
For $1\leq i\leq n$, 
$$
\hspace{-3.8cm}\frac{\partial\text{log}|T_{\omega}(Y)|^2}{\partial y_i}=
2\frac{\omega'(y_i)}{\omega(y_i)}+\sum^{n}_{j=1,j\neq i}\frac{4}{y_i-y_j}+\sum^{m}_{t=1}\frac{4(y_{i}-y_{n+t})}{(y_i-y_{n+t})^2+\mu^2_{t}},
$$
$$
\hspace{1.3cm} =2\frac{\omega'(y_i)}{\omega(y_i)}+\sum^{n}_{j=1,j\neq i}\frac{4}{y_i-y_j}+\sum^{m}_{t=1}\frac{4}{y_i-(y_{n+t}+i\mu_{t})}-\sum^{m}_{t=1}\frac{4i\mu_{t}}{(y_i-y_{n+t})^2+\mu^2_{t}},
$$
notice that the sum of the first three terms on the right hand side equals $0$ at $\{x_1,\cdots,x_{n+m}\}$ due to $(2.1)$ and $(2.9)$, which implies 
\begin{equation}
\sum^{m}_{t=1}\frac{\mu_{t}}{(x_i-x_{n+t})^2+\mu^2_{t}}=0,
\end{equation}
since the left side is real.\\
For $n+1\leq i\leq n+m$,
$$
\frac{\partial\text{log}|T_{\omega}(Y)|^2}{\partial y_i}=
\frac{\omega'(y_{i}+i\mu_{i-n})}{\omega(y_{i}+i\mu_{i-n})}+\frac{\omega'(y_{i}-i\mu_{i-n})}{\omega(y_{i}-i\mu_{i-n})}+\sum^{m}_{l=1,l\neq i-n}\frac{4(y_{i}-y_{n+l})}{(y_{i}-y_{n+l})^2+(\mu_{i-n}-\mu_{l})^2}+\sum^{n}_{s=1}\frac{4(y_i-y_s)}{(y_{i}-y_{s})^2+\mu^2_{i-n}}
$$
$$
\hspace{0.9cm} =2\frac{\omega'(y_{i}+i\mu_{i-n})}{\omega(y_{i}+i\mu_{i-n})}+\sum^{m}_{l=1,l\neq i-n}\frac{4}{(y_{i}+i\mu_{i-n})-(y_{n+l}+i\mu_{l})}+\sum^{n}_{s=1}\frac{4}{(y_{i}+i\mu_{i-n})-y_s}
$$
$$
\hspace{2.4cm} +\frac{\omega'(y_{i}-i\mu_{i-n})}{\omega(y_{i}-i\mu_{i-n})}-\frac{\omega'(y_{i}+i\mu_{i-n})}{\omega(y_{i}+i\mu_{i-n})}+\sum^{m}_{l=1,l\neq i-n}\frac{4i(\mu_{i-n}-\mu_{l})}{(y_{i}-y_{n+l})^2+(\mu_{i-n}-\mu_{l})^2}+\sum^{n}_{s=1}\frac{4i\mu_{i-n}}{(y_{i}-y_{s})^2+\mu^2_{i-n}}
,$$
again we find that the sum of the first three terms on the right hand side equals $0$ at $\{x_1,\cdots,x_{n+m}\}$, thus implies
\begin{equation}
\frac{\omega'(x_{i}-i\mu_{i-n})}{\omega(x_{i}-i\mu_{i-n})}-\frac{\omega'(x_{i}+i\mu_{i-n})}{\omega(x_{i}+i\mu_{i-n})}+\sum^{m}_{l=1,l\neq i-n}\frac{4i(\mu_{i-n}-\mu_{l})}{(x_{i}-x_{n+l})^2+(\mu_{i-n}-\mu_{l})^2}+\sum^{n}_{s=1}\frac{4i\mu_{i-n}}{(x_{i}-x_{s})^2+\mu^2_{i-n}}=0.
\end{equation}
Therefore we have shown that $\{x_1,\cdots,x_{n+m}\}$ is a critical point of $|T_{\omega}(Y)|^2$. 

The Hessian matrix $H$ of $(-\text{log}|T_{\omega}(Y)|^2)$ has four types off-diagonal elements and two types diagonal elements:
\begin{eqnarray*}
H_{ij} &=& -\frac{4}{(y_i-y_j)^2}, \quad 1\leq i\leq n, 1\leq j\leq n, i\neq j, \\
H_{ij} &=& -\frac{4[(y_i-y_j)^2-\mu^2_{j-n}]}{[(y_i-y_j)^2+\mu^2_{j-n}]^2}, \quad 1\leq i\leq n, n+1\leq j\leq n+m, \\
H_{ij} &=& -\frac{4[(y_i-y_j)^2-\mu^2_{i-n}]}{[(y_i-y_j)^2+\mu^2_{i-n}]^2}, \quad n+1\leq i\leq n+m, 1\leq j\leq n, \\
H_{ij} &=& -\frac{4[(y_i-y_j)^2-(\mu_{i-n}-\mu_{j-n})^2]}{[(y_i-y_j)^2+(\mu_{i-n}-\mu_{j-n})^2]^2}, \quad n+1\leq i\leq n+m, n+1\leq j\leq n+m, i\neq j,
\end{eqnarray*}
and
\begin{eqnarray*}
H_{ii} &=& \bigg(-\frac{2\omega'(y_i)}{\omega(y_i)}\bigg)'+\sum^{n}_{j=1,j\neq i}\frac{4}{(y_i-y_j)^2}+\sum^{m}_{t=1}\frac{4[(y_i-y_{n+t})^2-\mu^2_{t}]}{[(y_i-y_{n+t})^2+\mu^2_{t}]^2}, \quad 1\leq i\leq n, \\
H_{ii} &=& \bigg(-\frac{\omega'(y_i+i\mu_{i-n})}{\omega(y_i+i\mu_{i-n})}\bigg)'+\bigg(-\frac{\omega'(y_i-i\mu_{i-n})}{\omega(y_i-i\mu_{i-n})}\bigg)'+\sum^{m}_{l=1,l\neq i}\frac{4[(y_i-y_{n+l})^2-(\mu_{i-n}-\mu_l)^2]}{[(y_i-y_{n+l})^2+(\mu_{i-n}-\mu_l)^2]^2} \\
& & +\sum^{n}_{s=1}\frac{4[(y_i-y_{s})^2-\mu^2_{i-n}]}{[(y_i-y_{s})^2+\mu^2_{i-n}]^2}, \quad n+1\leq i\leq n+m.
\end{eqnarray*}
In order to find the condition for $H$ to be positive definite, it should be satisfied that
$$
H_{ii}>0 \quad\text{and}\quad H_{ii}>\sum^{n+m}_{j=1,j\neq i}|H_{ij}|, \quad 1\leq i\leq n+m,
$$
which is equivalent as
$$
H_{ii}>\sum^{n+m}_{j=1,j\neq i}|H_{ij}|, \quad 1\leq i\leq n+m.
$$
When $1\leq i\leq n$, we have
$$
\hspace{-0.8cm}H_{ii}-\sum^{n+m}_{j=1,j\neq i}|H_{ij}| = \bigg(-\frac{2\omega'(y_i)}{\omega(y_i)}\bigg)'+\sum^{m}_{t=1}\frac{4[(y_i-y_{n+t})^2-\mu^2_{t}]}{[(y_i-y_{n+t})^2+\mu^2_{t}]^2}-\sum^{m}_{t=1}\bigg|\frac{4[(y_i-y_{n+t})^2-\mu^2_{t}]}{[(y_i-y_{n+t})^2+\mu^2_{t}]^2}\bigg|
$$
$$
\hspace{1.3cm} =\bigg(-\frac{2\omega'(y_i)}{\omega(y_i)}\bigg)'+\sum^{m}_{t=1}\bigg[\frac{4}{(y_i-y_{n+t})^2+\mu^2_{t}}-\frac{8\mu^2_{t}}{[(y_i-y_{n+t})^2+\mu^2_{t}]^2}\bigg]
$$
$$
\hspace{3.5cm} -\sum^{m}_{t=1}\bigg|\frac{4}{(y_i-y_{n+t})^2+\mu^2_{t}}-\frac{8\mu^2_{t}}{[(y_i-y_{n+t})^2+\mu^2_{t}]^2}\bigg|
$$
$$
\hspace{-1.6cm} \geq\bigg(-\frac{2\omega'(y_i)}{\omega(y_i)}\bigg)'-\sum^{m}_{t=1}\frac{16\mu^2_{t}}{[(y_i-y_{n+t})^2+\mu^2_{t}]^2}
$$
$$
\hspace{-2.6cm} \geq\bigg(-\frac{2\omega'(y_i)}{\omega(y_i)}\bigg)'-\sum^{m}_{t=1}\frac{4}{(y_i-y_{n+t})^2},
$$
when $n+1\leq i\leq n+m$,
$$
\hspace{2.2mm} H_{ii}-\sum^{n+m}_{j=1,j\neq i}|H_{ij}| = \bigg(-\frac{\omega'(y_i+i\mu_{i-n})}{\omega(y_i+i\mu_{i-n})}\bigg)'+\bigg(-\frac{\omega'(y_i-i\mu_{i-n})}{\omega(y_i-i\mu_{i-n})}\bigg)'+\sum^{m}_{l=1,l\neq i-n}\frac{4[(y_i-y_{n+l})^2-(\mu_{i-n}-\mu_l)^2]}{[(y_i-y_{n+l})^2+(\mu_{i-n}-\mu_l)^2]^2}
$$
$$
\hspace{0.5cm} +\sum^{n}_{s=1}\frac{4[(y_i-y_{s})^2-\mu^2_{i-n}]}{[(y_i-y_{s})^2+\mu^2_{i-n}]^2}-\sum^{m}_{l=1,l\neq i-n}\bigg|\frac{4[(y_i-y_{n+l})^2-(\mu_{i-n}-\mu_l)^2]}{[(y_i-y_{n+l})^2+(\mu_{i-n}-\mu_l)^2]^2}\bigg|-\sum^{n}_{s=1}\bigg|\frac{4[(y_i-y_{s})^2-\mu^2_{i-n}]}{[(y_i-y_{s})^2+\mu^2_{i-n}]^2}\bigg|
$$
$$
\hspace{2.5cm} \geq\bigg(-\frac{\omega'(y_i+i\mu_{i-n})}{\omega(y_i+i\mu_{i-n})}\bigg)'+\bigg(-\frac{\omega'(y_i-i\mu_{i-n})}{\omega(y_i-i\mu_{i-n})}\bigg)'-\sum^{m}_{l=1,l\neq i-n}\frac{4}{(y_i-y_{n+l})^2}-\sum^{n}_{s=1}\frac{4}{(y_i-y_s)^2}
$$
$$
\hspace{0.2cm} =\bigg(-\frac{\omega'(y_i+i\mu_{i-n})}{\omega(y_i+i\mu_{i-n})}\bigg)'+\bigg(-\frac{\omega'(y_i-i\mu_{i-n})}{\omega(y_i-i\mu_{i-n})}\bigg)'-\sum^{n+m}_{j=1,j\neq i}\frac{4}{(y_i-y_j)^2}.
$$
The above computations show that if it holds that
\begin{equation}
\bigg(-\frac{2\omega'(x_i)}{\omega(x_i)}\bigg)'>\sum^{n+m}_{j=n+1}\frac{4}{(x_i-x_j)^2}
\end{equation}
and
\begin{equation}
\bigg(-\frac{\omega'(x_i+i\mu_{i-n})}{\omega(x_i+i\mu_{i-n})}\bigg)'+\bigg(-\frac{\omega'(x_i-i\mu_{i-n})}{\omega(x_i-i\mu_{i-n})}\bigg)'>\sum^{n+m}_{j=1,j\neq i}\frac{4}{(x_i-x_j)^2}
\end{equation}
thus $|T_{\omega}(Y)|^2$ has a (local) maximum value at $X=\{x_1,\cdots,x_n,x_{n+1},\cdots,x_{n+m}\}$. According to these computations we give the following theorem which provide a sufficient condition for $|T_{\omega}(Y)|^2$ to obtain its maximum value at $X$.

\par
\begin{thm}
Let $\omega(x)=\hat{\omega}(x)p(x)$, where $\hat{\omega}(x)$ takes an exceptional weight, namely $\omega(x)=e^{-x^2}/\eta^{2}_{H}(x)$ for exceptional Hermite polynomials, or $\omega(x)=x\cdot x^{\alpha}e^{-x}/\eta^{2}_{L}(x)$ for exceptional Laguerre polynomials, or $\omega(x)=(1-x^2)\cdot (1-x)^{\alpha}(1+x)^{\beta}/\eta^{2}_{J}(x)$ for exceptional Jacobi polynomials. If the denominators $\eta_{H}(x)$, $\eta_{L}(x)$, $\eta_{J}(x)$ satisfy the following conditions
\begin{eqnarray}
\big(\text{log}\eta_{\alpha}(x)\big)''+k_{\alpha} &\geq& 0, \hspace{2mm} x\in I \\
\big(\text{log}\eta_{\alpha}(x)\big)''|_{x=z_i}+k_{\alpha} &>& \sum^{n+m}_{j=n+1}\frac{1}{(x_i-x_j)^2},
\hspace{2mm} 1\leq i\leq n+m.
%\big(\text{log}\eta_{a}(x_i+i\mu_{i-n})\big)'' &+& \big(\text{log}\eta_{a}(x_i-i\mu_{i-n})\big)'' + 2k_a  >\sum^{n+m}_{j=1,j\neq i}\frac{2}{(x_i-x_j)^2}, \hspace{2mm} n+1\leq i\leq n+m,
\end{eqnarray}
then in the domain $I$ with respect to $\hat{\omega}(x)$, the real-valued function $|T_{\omega}(Y)|^2$ attains its maximum value at $X$, where $\alpha$ can be replaced by $H$, $L$, $J$, respectively, $k_H=1$, $k_L=k_J=0$.
\end{thm}

\begin{proof}
Till now it has been known that if $(4.3)$ and $(4.4)$ are satisfied then $|T_{\omega}(Y)|^2$ has a (local) maximum value at $X$. The uniqueness requires that $(\text{log}\omega(x))''=(q(x)/p(x))'\leq0$, $\forall x\in I$, provided the information of $p(x)$, $q(x)$ are given in $(2.3)$, $(2.4)$, $(2.5)$, respectively, we check case by case using the same notations as we did in the proof of Theorem 4.3. 
Assuming it satisfies that
\begin{eqnarray}
\big(\text{log}\eta_{H}(x)\big)''+1 &\geq& 0, \hspace{2mm} x\in(-\infty,\infty),\\
\big(\text{log}\eta_{L}(x)\big)'' &\geq& 0, \hspace{2mm} x\in(0,\infty),\\
\big(\text{log}\eta_{J}(x)\big)'' &\geq& 0, \hspace{2mm} x\in(-1,1),
\end{eqnarray}
then for exceptional Hermite polynomials, it holds that
$$
F(x)=2+2\big(\text{log}\eta_{H}(x)\big)''\geq 0 ,
$$
as well as for exceptional Laguerre polynomials $(\alpha>-1, k'>0)$ we have
$$
F(x)=\alpha+k'+1+2x^2\big(\text{log}\eta_{L}(x)\big)''\geq 0 ,
$$
and for exceptional Jacobi polynomials $(\alpha,\beta>-1, k'_1+k'_2>0)$ we have
$$
F(x)=(\alpha+\beta+2k'_1+2)(1+x^2)-(\beta-\alpha-2k'_2)2x+2(1-x^2)^2\big(\text{log}\eta_{J}(x)\big)''\geq 0,
$$
Moreover, since it holds respectively for the weight functions of exceptional Hermite, Laguerre, Jacobi polynomial that
$$
\bigg(-\frac{\omega'(x)}{\omega(x)}\bigg)'=
\begin{cases}
2+2\big(\text{log}\eta_{H}(x)\big)'' \\
\frac{\alpha+k'+1}{x^2}+2\big(\text{log}\eta_{L}(x)\big)'' >2\big(\text{log}\eta_{L}(x)\big)'' \\
\frac{(\alpha+\beta+2k'_1+2)(1+x^2)-(\beta-\alpha-2k'_2)2x}{(1-x^2)^2}+2\big(\text{log}\eta_{J}(x)\big)'' >2\big(\text{log}\eta_{J}(x)\big)''
\end{cases}
$$
we can rewrite $(4.3)$ and $(4.4)$ as stronger conditions which are implied by $(4.6)$:
\begin{equation}
\begin{cases}
1+\big(\text{log}\eta_{H}(x_i)\big)'' \\
\big(\text{log}\eta_{L}(x_i)\big)'' \\
\big(\text{log}\eta_{J}(x_i)\big)''
\end{cases}
\hspace{-4mm}>\sum^{n+m}_{j=n+1}\frac{1}{(x_i-x_j)^2}, \hspace{2mm} 1\leq i\leq n
\end{equation}
and
\begin{equation}
\begin{cases}
2+\big(\text{log}\eta_{H}(x_i+i\mu_{i-n})\big)''+\big(\text{log}\eta_{H}(x_i-i\mu_{i-n})\big)''\\
\big(\text{log}\eta_{L}(x_i+i\mu_{i-n})\big)''+\big(\text{log}\eta_{L}(x_i-i\mu_{i-n})\big)''\\
\big(\text{log}\eta_{J}(x_i+i\mu_{i-n})\big)''+\big(\text{log}\eta_{J}(x_i-i\mu_{i-n})\big)''
\end{cases}
\hspace{-4mm}>\sum^{n+m}_{j=1,j\neq i}\frac{2}{(x_i-x_j)^2}, \hspace{2mm} n+1\leq i\leq n+m.
\end{equation}

\par
Recall that $\omega(x)$ decays quickly at the boundary, thus $|T_{\omega}(Y)|^2$ tends to zero at the boundary. Concludingly, $|T_{\omega}(Y)|^2$ has a unique maximum at $X$ if $(4.5)$, $(4.6)$ are satisfied.
\end{proof}

\par
Notice that in $(4.1)$ and $(4.2)$ we have
$$
\frac{\mu_{t}}{(x_i-x_{n+t})^2+\mu^2_{t}}=\frac{1}{2i}\bigg[\frac{1}{(x_i-x_{n+t})-i\mu_{t}}-\frac{1}{(x_i-x_{n+t})+i\mu_{t}}\bigg],
$$
and
\begin{eqnarray*}
\frac{4i(\mu_{i-n}-\mu_{l})}{(x_{i}-x_{n+l})^2+(\mu_{i-n}-\mu_{l})^2} &=& 2\bigg[\frac{1}{(x_{i}-x_{n+l})-i(\mu_{i-n}-\mu_{l})}-\frac{1}{(x_{i}-x_{n+l})+i(\mu_{i-n}-\mu_{l})}\bigg], \\
\frac{4i\mu_{i-n}}{(x_{i}-x_{s})^2+\mu^2_{i-n}} &=& 2\bigg[\frac{1}{(x_{i}-x_{s})-i\mu_{i-n}}-\frac{1}{(x_{i}-x_{s})+i\mu_{i-n}}\bigg],
\end{eqnarray*}
thus obtain the following result.
\begin{cor}
If an exceptional orthogonal polynomial $P_{n+m}(z)$ has $n+m$ simple zeros consisting of $n$ real zeros and $m$ complex zeros:
$$
z_1=x_1,\ldots,z_n=x_n,z_{n+1}=x_{n+1}+i\mu_{1},\ldots,z_{n+m}=x_{n+m}+i\mu_{m}, 
$$
where $x_{i}\in\mathbb{R},i=1,\ldots,n+m$, $\mu_{j}\in\mathbb{R}, j=1,\ldots,m$,
then it satisfies
\begin{equation}
\sum^{m}_{t=1}\frac{1}{(x_i-x_{n+t})+i\mu_{t}}=\sum^{m}_{t=1}\frac{1}{(x_i-x_{n+t})-i\mu_{t}}, \quad 1\leq i\leq n,
\end{equation}
and 
\begin{eqnarray}
&&\frac{\omega'(x_{i}+i\mu_{i-n})}{\omega(x_{i}+i\mu_{i-n})}
+\sum^{m}_{l=1,l\neq i-n}\frac{2}{(x_{i}-x_{n+l})+i(\mu_{i-n}-\mu_{l})}
+\sum^{n}_{s=1}\frac{2}{(x_{i}-x_{s})+i\mu_{i-n}}\\
&=&\frac{\omega'(x_{i}-i\mu_{i-n})}{\omega(x_{i}-i\mu_{i-n})}
+\sum^{m}_{l=1,l\neq i-n}\frac{2}{(x_{i}-x_{n+l})-i(\mu_{i-n}-\mu_{l})}
+\sum^{n}_{s=1}\frac{2}{(x_{i}-x_{s})-i\mu_{i-n}}, \quad n+1\leq i\leq n+m. \nonumber
\end{eqnarray}
\end{cor}
In particular,  notice that
$$
\frac{\omega'(x_{i}+i\mu_{i-n})}{\omega(x_{i}+i\mu_{i-n})}=\frac{\hat{\omega}'(x_{i}+i\mu_{i-n})}{\hat{\omega}(x_{i}+i\mu_{i-n})}+\frac{p'(x_{i}+i\mu_{i-n})}{p(x_{i}+i\mu_{i-n})}=-2S_{1,i},
$$
it follows that the left hand side of $(4.13)$ is $0$, consequently implies
$$
\frac{\omega'(x_{i}-i\mu_{i-n})}{\omega(x_{i}-i\mu_{i-n})}
+\sum^{m}_{l=1,l\neq i-n}\frac{2}{(x_{i}-x_{n+l})-i(\mu_{i-n}-\mu_{l})}
+\sum^{n}_{s=1}\frac{2}{(x_{i}-x_{s})-i\mu_{i-n}}=0, \quad n+1\leq i\leq n+m.
$$

\section*{Acknowledgements}
\par
The author thanks to Satoshi Tsujimoto for fruitful discussions and helpful advices. This research did not receive any specific grant from funding agencies in the public, commercial, or not-for-profit sectors.





%\section*{References}

\end{document}